\documentclass[12pt]{amsart}
\usepackage[utf8]{inputenc}
\usepackage[margin=1in]{geometry}
\usepackage{amsfonts}
\usepackage{amsmath}
\usepackage{amssymb}
\usepackage{amsthm}
\usepackage{xcolor}
\usepackage{bm}
\usepackage{mathtools}
\usepackage{enumitem}
\usepackage{verbatim}
\usepackage{pinlabel}
\usepackage{caption}
\usepackage{subcaption}
\usepackage{tikz}
\usepackage{graphicx}
\usepackage{hyperref}
\usepackage{tikz-cd}
\usepackage{float}
\usepackage{geometry}
\usepackage{soul}
\usepackage{mathtools}

\newtheorem{theorem}{Theorem}[section]

\newtheorem{lemma}[theorem]{Lemma}
\newtheorem{proposition}[theorem]{Proposition}

\newtheorem{fact}[theorem]{Fact}

\newcommand{\nl}{\mathrm{(NL)}}
\newcommand{\Z}{\mathbb{Z}}

\usepackage{tcolorbox} 
\newlist{todolist}{itemize}{2}
\setlist[todolist]{label=$\square$}
\usepackage{pifont}

\title{Characterizing Property $\nl$ in Coxeter Groups} 
\author{Sahana Balasubramanya, Rachel Niebler, and Roberta Shapiro}

\date{}

\begin{document}

\begin{abstract}
    A group has Property $\nl$ -- which stands for ``no loxodromics" -- if no element of the group acts loxodromically on any hyperbolic space. In this brief note, we provide a complete characterization of which Coxeter groups have Property $\nl$. 
\end{abstract}
\maketitle

\vspace{-1.2cm}

\section{Introduction}

A group $G$ has Property $\nl$, which stands for \emph{no loxodromics}, if no element of $G$ acts loxodromically for any action (by isometries) on any hyperbolic space. Property $\nl$ was first formally introduced by the first author, Fournier-Facio, Genevois, and Sisto in \cite{PropNL}. 

In this paper, which is an extension of the REU project by the authors and Burkhalter \cite{BBNS}, our goal is to characterize exactly which Coxeter groups have Property (NL). Coxeter groups appear naturally as groups of symmetries of certain spaces and are defined in terms of an underlying graph (see Section~\ref{sec:background} for details). 

In \cite{BBNS}, the authors characterized Property $\nl$ in Right-Angled Coxeter groups (RACGs) and triangle groups (Coxeter groups with three generators). In this paper, we provide a complete classification of the property among \emph{all} finitely generated Coxeter groups.

\begin{theorem}\label{thm:main}
    A finitely generated Coxeter group has Property $\nl$ if and only if it is the direct product of spherical (finite) or affine Coxeter groups other than $D_\infty.$ 
\end{theorem}

We remark that the above theorem can likely be deduced by experts working with Coxeter groups with relative ease. However, to the best of our knowledge, this result has not been recorded in the literature. 

Moreover, it is a popular avenue of research to study which properties of a Coxeter group can be completely determined from their defining graph. Since the irreducible factors of a finitely generated Coxeter group can easily be determined from its defining graph, and the finite and irreducible affine groups have been completely classified, our theorem also implies that Property $\nl$ can be completely (and easily) determined from the defining graph of such a Coxeter group. \\

\textbf{Acknowledgments.} The authors would like to thank Anthony Genevois for suggesting this result and proof ideas. They would also like to thank the REU program held at the Georgia Institute of Technology in 2023 under the supervision of Dan Margalit, where the authors first met.

\section{Background}\label{sec:background}

A Coxeter group is defined via a \emph{Coxeter graph}, which is a graph $\Gamma$ that encodes a presentation for the group. The vertices of $\Gamma$ correspond to generators of the group of order 2. The edges of such a graph are labeled by elements of the set $\{3, 4, 5, \cdots \} \cup \{\infty\}$;  edges (or the lack thereof) correspond to relations in the group presentation in the following way.

\begin{itemize}[itemsep=5pt]
    \item If vertices $a$ and $b$ are not connected by an edge, then $(ab)^2=1$ (which implies $ab = ba$, so the corresponding generators commute). 
    \item If vertices $a$ and $b$ are connected by an edge labeled by an integer $m,$ then $(ab)^m=1$, but $(ab)^i\neq 1$ for all $0<i<m.$
    \item If generators $a$ and $b$ are connected by an edge labeled by $\infty,$ there is no added relation (also sometimes referred to as ``no relation"). 
\end{itemize}

Given such a graph $\Gamma,$ the associated Coxeter group is denoted $W_\Gamma$. A Coxeter group $G$ is finitely generated if and only if there is a finite Coxeter graph $\Gamma$ so that $G\cong W_\Gamma.$ In this paper, we will focus only on finitely generated Coxeter groups, so without loss of generality, all the graphs $\Gamma$ we consider are finite. 

From the description of the edge labels above, it follows that if $\Gamma$ is a disconnected graph, then $W_\Gamma$ can be expressed as a direct product of the Coxeter groups defined by the maximal connected components of $\Gamma$. A Coxeter group is \emph{irreducible} if it cannot be expressed as the direct product of more than one Coxeter group; equivalently, its defining Coxeter graph is connected. The following can then be easily deduced (see \cite[Section 6]{Humphreys} for instance). 

\begin{fact}\label{fact}
    Any finitely generated Coxeter group decomposes into a direct product of irreducible Coxeter groups. 
\end{fact}

Irreducible Coxeter groups can in turn be classified by the spaces of which they are the symmetries. An irreducible \emph{spherical} Coxeter group is a Coxeter group that acts geometrically and irreducibly on a sphere in $\mathbb{R}^n$. An \emph{affine} irreducible Coxeter group is an infinite Coxeter group that has a representation as a discrete affine reflection group \cite{Humphreys}. Coxeter groups can be neither spherical, not affine, in which case they are referred to as \emph{indefinite type}.

\section{Proof of Theorem~\ref{thm:main}}

To prove our result, we will first reduce the theorem to the case of irreducible Coxeter groups (Lemma~\ref{lemma:irreducible}). We then show that irreducible Coxeter groups of indefinite type fail to have Property $\nl$ (Lemma~\ref{lemma:notaffinesphericalthenacylindricallyhyp}). We also show that spherical and irreducible affine Coxeter groups (other than $D_\infty$) have Property $\nl$ (Lemma~\ref{lemma:spherical} and Proposition~\ref{prop:affinenl}). These results will then be pieced together to prove the main theorem. 

We begin by reducing to the case of irreducible Coxeter groups. As explained in Section~\ref{sec:background} (see Fact~\ref{fact}), there is a canonical way to express a Coxeter group as a direct product of irreducible Coxeter groups. In this case, we have the following. 

\begin{lemma}\label{lemma:irreducible}
    A reducible Coxeter group $W_\Gamma$ has Property $\nl$ if and only if all its irreducible factors have Property $\nl$.
\end{lemma}

\begin{proof}
    Suppose, by contradiction, an irreducible factor $W_\Gamma'$ does not have Property $\nl$. Then $W_\Gamma'$ has an element $\gamma$ that acts loxodromically for an action on some hyperbolic space $X.$ Define an action of $W_\Gamma$ on $X$ by extending the action $W_\Gamma'$ by the identity to the rest of the group (so all generators in the graph $\Gamma \backslash \Gamma'$ act by the identity). This action
    retains $\gamma$ as a loxodromic element, which means $W_\Gamma$ does not have Property $\nl$. 
    
    Conversely, it follows from \cite[Proposition 4.1]{PropNL} that Property $\nl$ is preserved under direct sums, which are isomorphic to direct products when finitely many factors are involved.
\end{proof}

Irreducible spherical Coxeter groups are straightforward to work with thanks to the following well-known result; see \cite{paris} for instance. Elements in a finite group cannot act loxodromically on any space, since each element has finite order.

\begin{lemma}\label{lemma:spherical}
    Spherical Coxeter groups are finite. In particular, they have Property $\nl$.
\end{lemma}

We now consider irreducible Coxeter groups that are neither spherical nor affine. 

\begin{lemma}\label{lemma:notaffinesphericalthenacylindricallyhyp}
    Suppose $W_\Gamma$ is an irreducible Coxeter group of indefinite type. Then $W_\Gamma$ is acylindrically hyperbolic. In particular,  such groups fail to have Property $\nl$.
\end{lemma}

\begin{proof}
    Let $W_\Gamma$ be an irreducible Coxeter group that is neither spherical nor affine. By results of Davis \cite{davis98} and Moussong \cite{moussong}, such a group acts properly and cocompactly on a (proper) CAT(0) space (the Davis complex). By work of Caprace--Fujiwara \cite{CF}, this action has rank 1 elements. It follows by work of Sisto \cite[Theorems 1.1, 1.3]{sisto} that such groups have a proper, infinite, hyperbolically embedded subgroup. Applying the work of Dahmani-Guirardel-Osin \cite[Theorem 1.2]{Osin}, we conclude that $W_\Gamma$ is acylindrically hyperbolic.  

     Acylindrically hyperbolic groups by definition admit a non-elementary (acylindrical) action on a hyperbolic spaces, and therefore must contain (lots of) loxodromic elements. Thus they do not have Property $\nl$. 
\end{proof}

It remains to consider the case of irreducible affine Coxeter groups. We note that such Coxeter groups are completely classified, see for example \cite[Figure 1]{MollerVarghese}.

\begin{proposition}\label{prop:affinenl}
    Let $W_\Gamma$ be an irreducible affine Coxeter group. Then $W_\Gamma$ has Property $\nl$ if and only if it is not isomorphic to $D_\infty.$
\end{proposition}

To prove the proposition, we will use the following result about Property $\nl$.

\begin{lemma}\cite[Corollary 1.5]{PropNL} \label{lemma:surjnl}
    Let $G$ be a finitely generated amenable group. Then $G$ has Property $\nl$ if and only if it does not surject onto $\Z$ nor $D_\infty.$
\end{lemma}

Thus, we must prove the following two claims:
\begin{enumerate}
    \item irreducible affine Coxeter groups $W_\Gamma$ are amenable ; and
    \item irreducible affine Coxeter groups do not surject onto $\Z$ nor $D_\infty.$
\end{enumerate}

For the first claim, we have the following. 
\begin{lemma}\label{lemma:affineamenable}
    Suppose $W_\Gamma$ is an irreducible affine Coxeter group. Then $W_\Gamma$ is amenable.
\end{lemma}

\begin{proof}
    By \cite[Corollary 1.6]{qithesis}, an irreducible affine Coxeter group is virtually abelian. All virtually abelian groups are solvable, and therefore amenable.
\end{proof} 

For the second claim, we have the following. 

\begin{lemma}\label{lemma:nosurjection}
    Let $W_\Gamma$ be a finitely generated irreducible affine Coxeter group. Then \begin{enumerate} 
    \item there is no surjection of $W_\Gamma$ onto $\Z$; and  
    \item there is no surjection into $D_\infty$ unless $W_\Gamma \cong D_\infty$. 
    \end{enumerate}
\end{lemma}

\begin{proof}
   By \cite[Theorem 1.1]{MollerVarghese}, an irreducible affine Coxeter group $W_\Gamma$ is just infinite, meaning that the group is infinite and every proper quotient is finite. 

   Let $K = \Z$ or $D_\infty$. Suppose for contradiction that $\phi \colon W_\Gamma \to K$ is a surjective homomorphism. By the first isomorphism theorem, \[\dfrac{W_\Gamma}{ \ker(\phi)} \cong K.\]

   Since $W_\Gamma$ is just infinite and $K$ is infinite, it follows that $K$ is not a proper quotient. Thus either $K=\{1\}$ (a contradiction) or $K\cong W_\Gamma.$
   However, this is impossible if $K= \Z$ since $\Z$ is not a Coxeter group. In the case that $K = D_\infty$, it follows that $W_\Gamma \cong D_\infty$. 
\end{proof}

\begin{proof}[Proof of Proposition~\ref{prop:affinenl}]
    Let $W_\Gamma$ be an irreducible Coxeter group. Lemmas~\ref{lemma:affineamenable} and ~\ref{lemma:nosurjection} combine to show that $W_\Gamma$ is finitely generated, amenable and does not surject onto $\Z$ nor $D_\infty$ unless $W_\Gamma \cong D_\infty$. By Lemma~\ref{lemma:surjnl}, if $W_\Gamma \not\cong D_\infty$, then $W_\Gamma$ has Property $\nl$. Conversely, if $W_\Gamma$ has Property $\nl$, then $W_\Gamma \not\cong D_\infty$, since $D_\infty$ has an action on a line with translations, which are loxodromic elements. 
\end{proof}

\begin{proof}[Proof of Theorem~\ref{thm:main}]
    Let $W_\Gamma$ be a finitely generated Coxeter group with defining graph $\Gamma.$ By Lemma~\ref{lemma:irreducible}, we may focus on the  irreducible factor(s) of this group, and assume $W_\Gamma$ is irreducible. 

    Irreducible Coxeter groups are either spherical, affine, or neither. By Lemma~\ref{lemma:spherical}, all spherical irreducible Coxeter groups have Property $\nl$. By Lemma~\ref{lemma:notaffinesphericalthenacylindricallyhyp}, irreducible Coxeter groups that are neither spherical nor affine do not have Property $\nl$. By Proposition~\ref{prop:affinenl}, an irreducible affine Coxeter group satisfies Property $\nl$ if and only if $W_\Gamma \not\cong D_\infty$.  It follows then that $W_\Gamma$ must be spherical or affine, other than $D_\Gamma$. 

    The converse direction of the theorem follows from the above facts and Lemma~\ref{lemma:irreducible}. 
\end{proof}

\bibliographystyle{plainurl} 
\bibliography{bibliography}

\end{document}